\documentclass[11pt]{amsart}
\usepackage{amsfonts,bbm,mathrsfs,amssymb,amsmath,amsthm}
\newtheorem{proposition}{Proposition}
\newtheorem{lemma}{Lemma}
\newtheorem*{thmA}{Theorem A}
\newtheorem*{thmB}{Theorem B}
\newtheorem*{thmC}{Theorem C}
\newtheorem*{thmD}{Theorem D}
\newtheorem*{cconj}{Chern Conjecture}
\newtheorem*{mthm}{Main Theorem}
\newtheorem*{cconja}{Conjecture A}
\newtheorem*{cconjb}{Conjecture B}

\begin{document}

\title[A new characterization of the Clifford torus]{A new characterization of the Clifford torus via scalar curvature
pinching}
\thanks{Research supported by the Chinese NSF, Grant
No. 11071211; the Trans-Century Training Programme Foundation for
Talents by the Ministry of Education of China.}

\author{Hong-wei Xu}
\author{Zhi-yuan Xu}

\address{Center of Mathematical Sciences \\ Zhejiang University \\ Hangzhou 310027 \\
China} \email{xuhw@cms.zju.edu.cn; srxwing@zju.edu.cn}
\date{}

\numberwithin{equation}{section}

\maketitle

\begin{abstract}
Let $M^n$ be a compact hypersurface with constant mean curvature $H$
in $\mathbb{S}^{n+1}$. Denote by $S$ the squared norm of the second
fundamental form of $M$. We prove that there exists a positive
constant $\gamma(n)$ depending only on $n$ such that if
$|H|\leq\gamma(n)$ and $\beta(n,H)\leq
S\leq\beta(n,H)+\frac{n}{23}$, then $S\equiv\beta(n,H)$ and $M$ is
one of the following cases: (i)
$\mathbb{S}^{k}(\sqrt{\frac{k}{n}})\times
\mathbb{S}^{n-k}(\sqrt{\frac{n-k}{n}})$, $\,1\le k\le n-1$; (ii)
$\mathbb{S}^{1}(\frac{1}{\sqrt{1+\mu^2}})\times
\mathbb{S}^{n-1}(\frac{\mu}{\sqrt{1+\mu^2}})$. Here
$\beta(n,H)=n+\frac{n^3}{2(n-1)}H^2+\frac{n(n-2)}{2(n-1)}\sqrt{n^2H^4+4(n-1)H^2}$
and $\mu=\frac{n|H|+\sqrt{n^2H^2+4(n-1)}}{2}$. This provides a new
characterization of the Clifford torus.
\end{abstract}

\section{Introduction}\label{sec1}
\par Let $M^n$ be an $n$-dimensional compact
hypersurface with constant mean curvature in an $(n+1)$-dimensional
unit sphere $\mathbb{S}^{n+1}$. Denote by $R$, $H$ and $S$ the
scalar curvature, the mean curvature and the squared norm of the
second fundamental form of $M$, respectively. It follows from the
Gauss equation that $R=n(n-1)+n^2H^{2}-S$ . A famous rigidity
theorem due to  Simons, Lawson, and Chern, do Carmo and Kobayashi
(\cite{CDK}, \cite{L1}, \cite{S}) says that if $M$ is a closed
minimal hypersurface in $\mathbb{S}^{n+1}$ satisfying $S\leq n$,
then $S\equiv0$ and $M$ is the great sphere $\mathbb{S}^{n}$, or
$S\equiv n$ and $M$ is the Clifford torus
$\mathbb{S}^{k}(\sqrt{\frac{k}{n}})\times
\mathbb{S}^{n-k}(\sqrt{\frac{n-k}{n}})$, $\,1\le k\le n-1$.
Afterward, Li-Li \cite{LL} improved Simons' pinching constant for
$n$-dimensional closed minimal submanifolds in $\mathbb{S}^{n+p}$ to
$\mathrm{max}\{\frac{n}{2-1/p},\frac{2}{3}n\}$. Further developments
on this rigidity theorem have been made by many other authors (see
\cite{CY}, \cite{D}, \cite{L}, \cite{X1}, \cite{X2}, \cite{Y1},
etc.). In 1970's, Chern proposed the following conjecture.

\begin{cconj}\label{cj} {\rm (See \cite{CDK}, \cite{Y2})}
Let $M$ be a compact minimal hypersurface in the unit sphere $\mathbb{S}^{n+1}$.\\
(A) If $S$ is constant, then the possible values of $S$ form a
discrete set. In particular, if $n\leq
S\leq2n$, then $S=n$, or $S=2n$.\\
(B) If $n\leq S\leq 2n$, then $S\equiv n$, or $S\equiv 2n$.
\end{cconj}

\par In 1983, Peng and Terng (\cite{PT1}, \cite{PT2}) made a breakthrough on the Chern
conjecture, and proved the following

\begin{thmA}\label{thm1}
Let $M$ be a compact minimal hypersurface in the unit sphere $\mathbb{S}^{n+1}$.\\
(i) If $S$ is constant, and if $n\leq S\leq n+\frac{1}{12n}$, then
$S=n$.\\
(ii) If $n\le 5$, and if
$n\leq S\leq n+\tau_1(n)$, where
$\tau_1(n)$ is a positive constant depending only on $n$, then
$S\equiv n$.
\end{thmA}

During the past three decades, there have been some important
progress on these aspects(see \cite{C},  \cite{CI}, \cite{Cheng},
\cite{DX}, \cite{SY}, \cite{WX}, \cite{XU1}, \cite{YC2}, \cite{Z1},
etc.). In 1993, Chang \cite{C} proved Chern Conjecture (A) in
dimension three. Yang-Cheng \cite{YC2} improved the pinching
constant $\frac{1}{12n}$ in Theorem A(i) to $\frac{n}{3}$. Later,
Suh-Yang \cite{SY} improved this pinching constant to
$\frac{3}{7}n$.

\par In 2007, Wei and
Xu \cite{WX} proved that if $M$ is a compact minimal hypersurface in
$\mathbb{S}^{n+1}$, $n=6, 7$, and if $n\leq S\leq n+ \tau_2(n)$,
where $\tau_2(n)$ is a positive constant depending only on $n$, then
$S\equiv n$. Later, Zhang \cite{Z1} extended the second pinching
theorem due to Peng-Terng \cite{PT2} and Wei-Xu \cite{WX} to the
case of $n=8$. Recently, Ding and Xin \cite{DX} obtained the
striking result, as stated

\begin{thmB}\label{thm1}
Let $M$ be an $n$-dimensional compact minimal hypersurface in the
unit sphere $\mathbb{S}^{n+1}$, and $S$ the squared length of the
second fundamental form of $M$. If $n\geq6$, and if $n\leq S\leq
\frac{24}{23}n$, then $S\equiv n$, i.e., $M$ is a Clifford torus.
\end{thmB}

\par The rigidity problem for hypersurfaces of constant
mean curvature is much more complicated than the minimal
hypersurface case (see \cite{AL}, \cite{B}, \cite{C2}, \cite{X0},
\cite{X1}, \cite{XT}, \cite{XX}). For example, the famous Lawson
conjecture \cite{L2}, verified by Brendle \cite{B}, states that the
Clifford torus is the only compact embedded minimal surface with
genus $1$ in $\mathbb{S}^3$. Recently, Andrews and Li \cite{AL}
proved a beautiful classification theorem for constant mean
curvature tori in $\mathbb{S}^3$, which implies that constant mean
curvature tori in $\mathbb{S}^3$ include the Clifford torus as well
as  many other CMC surfaces.

\par Let $M$ be an $n$-dimensional compact submanifold with
parallel mean curvature in the unit sphere $\mathbb{S}^{n+p}$. We
put
$$\alpha(n,H)=n+ \frac{n^{3}}{2(n-1)}H^{2} -
\frac{n(n-2)}{2(n-1)}\sqrt{n^{2}H^{4}+4(n-1)H^{2}},$$
$$\beta(n,H)=n+\frac{n^3}{2(n-1)}H^2+\frac{n(n-2)}{2(n-1)}\sqrt{n^2H^4+4(n-1)H^2},$$
$$\lambda=\frac{n|H|+\sqrt{n^2H^2+4(n-1)}}{2(n-1)},\,\,\,\mu=\frac{n|H|+\sqrt{n^2H^2+4(n-1)}}{2}.$$

\par In \cite{X0,X1}, the first author proved the generalized
Simons-Lawson-Chern-do Carmo-Kobayashi theorem for submanifolds with
parallel mean curvature in a sphere.
\begin{thmC}
Let $M$ be an $n$-dimensional oriented compact submanifold with
parallel mean curvature in an $(n+p)$-dimensional unit sphere
$\mathbb{S}^{n+p}$. If $S\leq C(n,p,H),$ then $M$ is either a
totally umbilic sphere, a Clifford torus in an $(n+1)$-sphere, or
the Veronese surface in $\mathbb{S}^{4}(\frac{1} {\sqrt{1+H^{2}}})$.
In particular, if $M$ is a compact hypersurface with constant mean
curvature in $\mathbb{S}^{n+1}$, and if $S\leq \alpha(n,H),$  then
$M$ is either a totally umbilic sphere, or a Clifford torus. Here
the constant $C(n,p,H)$ is defined by
$$C(n,p,H)=\left\{\begin{array}{llll} \alpha(n,H),&\mbox{\ for\ } p=1, \mbox{\ or\ } p=2 \\ & \mbox{\ and\ }
 H\neq0,\\
\min\Big\{\alpha(n,H),\frac{1}{3}(2n+5nH^2)\Big\},&\mbox{\
otherwise.\ }
\end{array} \right.$$
\end{thmC}
\par The second pinching theorem for $n(\leq 8)$-dimensional
hypersurfaces with small constant mean curvature was proved for
$n\le 7$ by Cheng-He-Li \cite{CHL} and Xu-Zhao \cite{XZ}, and for
$n=8$ by Chen-Li \cite{CL} and Xu \cite{XU1}, respectively. In
\cite{XX}, the authors obtained the following second pinching
theorem for hypersurfaces with small constant mean curvature in
spheres.

\begin{thmD}\label{thm3}
Let $M$ be an $n$-dimensional compact hypersurface with constant
mean curvature in the unit sphere $\mathbb{S}^{n+1}$. There exist
two positive constants $\gamma_0(n)$ and $\delta_0(n)$ depending
only on $n$ such that if $|H|\leq\gamma_0(n)$, and $\beta(n,H)\leq
S\leq\beta(n,H)+\delta_0(n)$, then $S\equiv\beta(n,H)$ and $M$ is
one of the following cases: (i)
$\mathbb{S}^{k}(\sqrt{\frac{k}{n}})\times
\mathbb{S}^{n-k}(\sqrt{\frac{n-k}{n}})$, $\,1\le k\le n-1$;  (ii)
$\mathbb{S}^{1}(\frac{1}{\sqrt{1+\mu^2}})\times
\mathbb{S}^{n-1}(\frac{\mu}{\sqrt{1+\mu^2}})$.
\end{thmD}

Note that the second pinching constant $\delta_0(n)$ in Theorem D is
equal to $O(\frac{1}{n^2})$. The purpose of this paper is to give a
new characterization of the Clifford torus. We prove the following
theorem.

\begin{mthm}\label{mthm}
Let $M$ be an $n$-dimensional compact hypersurface with constant
mean curvature in the unit sphere $\mathbb{S}^{n+1}$. There exists a
positive constant $\gamma(n)$ depending only on $n$ such that if
$|H|\leq\gamma(n)$, and $\beta(n,H)\leq
S\leq\beta(n,H)+\frac{n}{23}$, then $S\equiv\beta(n,H)$ and $M$ is
one of the following cases: (i)
$\mathbb{S}^{k}(\sqrt{\frac{k}{n}})\times
\mathbb{S}^{n-k}(\sqrt{\frac{n-k}{n}})$, $\,1\le k\le n-1$;  (ii)
$\mathbb{S}^{1}(\frac{1}{\sqrt{1+\mu^2}})\times
\mathbb{S}^{n-1}(\frac{\mu}{\sqrt{1+\mu^2}})$.
\end{mthm}

\par Our main theorem generalizes Theorem B and improves Theorem
D as well. The key ingredient of our proof is to estimate $A-2B$. By
using the parameter method, we obtain an upper bound for $A-2B$ in
the form of $\frac{1}{3}[S+4+C(n)G^{1/3}+q(n,H)]|\nabla h|^2$. Then
we derive a new integral estimate with two parameters. By choosing
suitable values of the parameters, we show that $\nabla h=0$ which
implies that $M$ is a Clifford torus.

\section{Preliminaries}\label{sec2}
\par Let $M^n$ be an $n(\geq2)$-dimensional compact
hypersurface with constant mean curvature in the unit sphere
$\mathbb{S}^{n+1}$. We shall make use of the following convention on
the range of indices:
                     $$1\leq i, j, k, \ldots, \leq n.$$
We choose a local orthonormal frame $\{e_1, e_2, \ldots, e_{n+1}\}$
near a fixed point $x\in M$ over $\mathbb{S}^{n+1}$ such that
$\{e_1, e_2, \ldots, e_n\}$ are tangent to $M$.
\par Let $\{\omega_1, \omega_2,
\ldots, \omega_{n+1}\}$ be the dual frame fields of $\{e_1, e_2,
\ldots, e_{n+1}\}$. Denote by $R$, $H$, $h$ and $S$ the scalar
curvature, the mean curvature, the second fundamental form and the
squared length of the second fundamental form of $M$, respectively.
Then we have
\begin{equation}\label{2.2}
h=\sum\limits_{i,j}h_{ij}\omega_{i}\otimes\omega_{j},\,\,
h_{ij}=h_{ji},
\end{equation}
\begin{equation}\label{2.3}
S=\sum\limits_{i,j}h_{ij}^{2}, \,\,
H=\frac{1}{n}\sum\limits_{i}h_{ii},
\end{equation}
\begin{equation}\label{2.4}
R=n(n-1)+n^2H^{2}-S.
\end{equation}
Choose $e_{n+1}$ such that $H\ge0$. Denote by $h_{ijk}$, $h_{ijkl}$
and $h_{ijklm}$ the first, second and third covariant derivatives of
the second fundamental tensor $h_{ij}$, respectively. Then we have
\begin{equation}\label{2.5}
\nabla
h=\sum\limits_{i,j,k}h_{ijk}\omega_{i}\otimes\omega_{j}\otimes\omega_{k},\,\,
h_{ijk}=h_{ikj},
\end{equation}
\begin{equation}\label{2.6}
h_{ijkl}=h_{ijlk}+\sum\limits_{m}h_{mj}R_{mikl}+\sum\limits_{m}h_{im}R_{mjkl},
\end{equation}
\begin{equation}\label{2.7}
h_{ijklm}=h_{ijkml}+\sum\limits_{r}h_{rjk}R_{rilm}+\sum\limits_{r}h_{irk}R_{rjlm}+\sum\limits_{r}h_{ijr}R_{rklm}.
\end{equation}
Take a suitable orthonormal frame $\{e_1, e_2, \ldots, e_n\}$  at
$x$ such that $h_{ij}=\lambda_{i}\delta_{ij}$ for all $i$, $j$. Then
we have
\begin{equation}\label{2.8}
\frac{1}{2}\Delta S=S(n-S)-n^2H^2+nHf_3+|\nabla h|^2,
\end{equation}
\begin{eqnarray}\label{2.9}
\frac{1}{2}\Delta|\nabla h|^2&=&\nonumber (2n+3-S)|\nabla
h|^2-\frac{3}{2}|\nabla S|^2+|\nabla^2
h|^2\\
& &+3(2B-A)+3nHC,
\end{eqnarray}
where $$f_k=\sum\limits_{i}\lambda_{i}^k,\,\,
A=\sum\limits_{i,j,k}h_{ijk}^2\lambda_{i}^2,\,\,
B=\sum\limits_{i,j,k}h_{ijk}^2\lambda_{i}\lambda_{j},\,\,
C=\sum\limits_{i,j,k}h_{ijk}^2\lambda_{i}.$$\\
Following \cite{XX}, we have
\begin{equation}\label{2.11}
|\nabla^2 h|^2\geq\frac{3}{4}\sum\limits_{i\neq
j}t_{ij}^2=\frac{3}{4}\sum\limits_{i,j}t_{ij}^2,
\end{equation}
\begin{equation}\label{2.12}
3(A-2B)\leq \sigma S|\nabla h|^2,
\end{equation}
where $t_{ij}=h_{ijij}-h_{jiji}$ and $\sigma=\frac{\sqrt{17}+1}{2}$.
Moreover, we have
\begin{equation}\label{2.14}
\int_{M}(A-2B)dM=\int_{M}(nHf_3-S^2-f_{3}^2+Sf_4-\frac{|\nabla
S|^2}{4})dM.
\end{equation}
\par Denote by $\phi:=\sum\limits_{i,j}\phi_{ij}\omega_{i}\otimes\omega_{j}$ the trace free second fundamental form of $M$. By diagonalizing $h_{ij}$, we have $\phi_{ij}=\mu_i\delta_{ij}$,
where $\mu_i=\lambda_i-H$. Putting $\Phi=|\phi|^2$ and
$\bar{f}_k=\sum\limits_{i}\mu_i^k$, we get $\Phi=S-nH^2$ and
$f_3=\bar{f}_3+3H\Phi+nH^3$. From (\ref{2.8}), we obtain
\begin{equation}\label{3.1}
\frac{1}{2}\Delta\Phi=-F(\Phi)+|\nabla\phi|^2,
\end{equation}
where $F(\Phi)=\Phi^2-n\Phi-nH^2\Phi-nH\bar{f}_3$. Therefore, we
have
\begin{equation}\label{3.2}
|\nabla\Phi|^2=\frac{1}{2}\Delta(\Phi)^2-\Phi\Delta\Phi=\frac{1}{2}\Delta(\Phi)^2+2\Phi
F(\Phi)-2\Phi|\nabla\phi|^2,
\end{equation}
\begin{equation}\label{3.3}
\int_MF(\Phi)dM=\int_M|\nabla\phi|^2dM.
\end{equation}
\par Put $\beta_0(n,H)=\beta(n,H)-nH^2.$
When $\Phi\geq\beta_0(n,H)$, it's seen from Proposition \ref{prop1}
(i) that
\begin{equation}\label{3.4}
F(\Phi)\geq0.
\end{equation}
Moreover, if $F(\Phi)=0$, then $\Phi=\beta_0(n,H)$.
\par Set
$G=\sum\limits_{i,j}(\lambda_i-\lambda_j)^2(1+\lambda_i\lambda_j)^2$.
Then we have
\begin{equation}\label{3.5}
G=2[Sf_4-f_3^2-S^2-S(S-n)+2nHf_3-n^2H^2].
\end{equation}
This together with (\ref{2.8}) and (\ref{2.14}) implies
\begin{equation}\label{3.6}
\frac{1}{2}\int_MGdM=\int_M[(A-2B)-|\nabla h|^2+\frac{1}{4}|\nabla
S|^2]dM.
\end{equation}

\par The following propositions will be used in the proof
of Main Theorem.

\begin{proposition}\label{prop1}
{\rm (See \cite{X1}, \cite{Z1})} (i) Let $a_1$, $a_2$, ..., $a_n$ be
real numbers satisfying $\sum\limits_{i}a_i=0$ and
$\sum\limits_{i}a_i^2=\mathscr{A}$.
Then$$|\sum\limits_{i}a_i^3|\leq\frac{n-2}{\sqrt{n(n-1)}}\mathscr{A}^{\frac{3}{2}},$$
and the equality holds if and only if at least $n-1$ numbers of
 $a_i$'s are same with each other.\\
(ii) Assume $f_n(t)=17[t-2-\eta(n)][3(n-2)t+(n+2)\eta(n)+10-4n]$ and
$g_n(t)=[8+16\eta(n)](4t-2-3\sqrt{-2t^2+2t+8})$, for $2\leq t\leq
\frac{1+\sqrt{17}}{2}$ and $4\leq n\leq5$. Here $\eta(4)=0.16$ and
$\eta(5)=0.23$. Then
$$h_n(t)=f_n(t)-g_n(t)\leq0.$$
\end{proposition}

\begin{proposition}\label{prop2}
Let $a_{ij}$, $b_i$ ($i,j=1,\dots,n$) be real numbers satisfying
$$\sum\limits_{i}b_i=0,\,\,\, \sum\limits_{i}b_{i}^{2}=\mathscr{B}(>0),\,\,\, \sum\limits_{i,j}(b_{i}+b_{j})a_{ij}=\mathscr{C}. $$  Then we have
$$\sum\limits_{i}a_{ii}^{2}+3\sum\limits_{i\neq
j}a_{ij}^{2}\geq\frac{3\mathscr{C}^2}{2(n+4)\mathscr{B}}.$$
\end{proposition}
\begin{proof}
Applying the Lagrange multiplier method, we compute the minimum
value $L_{min}$ of $L=\sum\limits_{i}a_{ii}^{2}+3\sum\limits_{i\neq
j}a_{ij}^{2}$ with constraints
$\sum\limits_{i}b_i=0$,\,$\sum\limits_{i}b_{i}^{2}=\mathscr{B}$,\,
$\sum\limits_{i,j}(b_{i}+b_{j})a_{ij}=\mathscr{C}.$ We consider the
function $L^*=\sum\limits_{i}a_{ii}^{2}+3\sum\limits_{i\neq
j}a_{ij}^{2}+\lambda\sum\limits_{i}b_i+\mu(\sum\limits_{i}b_{i}^{2}-\mathscr{B})
+\nu[\sum\limits_{i,j}(b_{i}+b_{j})a_{ij}-\mathscr{C}]$. At the
extreme points of $L$, we have
\begin{equation}\label{4.1}
\frac{\partial L^*}{\partial a_{ii}}=2a_{ii}+2\nu
b_{i}=0,\,\,\mbox{\ for\ }\,\,1\leq i\leq n,
\end{equation}
\begin{equation}\label{4.2}
\frac{\partial L^*}{\partial a_{ij}}=6a_{ij}+\nu(b_{i}+b_{j})=0,\,\,
\mbox{\ for\ }\,\, 1\leq i, j\leq n, \,\, i\neq j,
\end{equation}
\begin{equation}\label{4.3}
\sum\limits_{i}b_i=0,\,\,\,
\sum\limits_{i}b_{i}^{2}=\mathscr{B},\,\,\,
\sum\limits_{i,j}(b_{i}+b_{j})a_{ij}=\mathscr{C}.
\end{equation}
It follows from (\ref{4.1}) and (\ref{4.2}) that
\begin{equation}\label{4.4}
2\sum\limits_{i}a_{ii}^{2}+2\nu\sum\limits_{i}a_{ii}b_{i}=0,
\end{equation}
\begin{equation}\label{4.5}
6\sum\limits_{i\neq j}a_{ij}^{2}+\nu\sum\limits_{i\neq
j}a_{ij}(b_{i}+b_{j})=0.
\end{equation}
Hence
\begin{equation}\label{4.6}
2(\sum\limits_{i}a_{ii}^{2}+3\sum\limits_{i\neq j}a_{ij}^{2})=-\nu
\mathscr{C}.
\end{equation}
On the other hand, we have
\begin{equation}\label{4.7}
2\sum\limits_{i}b_{i}a_{ii}+2\nu\sum\limits_{i}b_{i}^{2}=0,
\end{equation}
\begin{equation}\label{4.8}
6\sum\limits_{i\neq j}(b_{i}+b_{j})a_{ij}+\nu\sum\limits_{i\neq
j}(b_{i}+b_{j})^{2}=0.
\end{equation}
Combining (\ref{4.3}), (\ref{4.7}) and (\ref{4.8}), we have
\begin{equation}\label{4.11}
\nu=-\frac{3\mathscr{C}}{(n+4)\mathscr{B}}.
\end{equation}
This together with (\ref{4.6}) and the fact that the Hessian matrix
of $L$ is identically positive definite implies that
\begin{equation}\label{4.12}
L_{min}=\frac{3\mathscr{C}^2}{2(n+4)\mathscr{B}}.
\end{equation}
\end{proof}

\section{Estimate for $A-2B$}\label{sec3}

\par It plays a crucial role to estimate for $A-2B$ in our work. Using an analogous argument as in the proof of Lemma 3.2 in
\cite{Z1}, we get the following
\begin{lemma}\label{lem3}
Let $M$ be an $n$-dimensional closed hypersurface with constant mean
curvature in the unit sphere $\mathbb{S}^{n+1}$. If
$\lambda_1^2-4\lambda_1\lambda_2\geq tS$, for some $t\in[2,
\frac{1+\sqrt{17}}{2}]$, then
$(\lambda_1^2-4\lambda_1\lambda_2)-(\lambda_1^2-4\lambda_1\lambda_i)\geq
rS$, for $i\neq1,2$, where $r=\frac{16t-8-12\sqrt{-2t^2+2t+8}}{17}$.
\end{lemma}

For the low dimensional cases, we give an up bound for $A-2B$.
\begin{lemma}\label{lem4}
If $4\leq n\leq5$, then
$$3(A-2B)\leq[2+\eta(n)]S|\nabla h|^2,$$ where $\eta(4)=0.16$ and $\eta(5)=0.23$.
\end{lemma}
\begin{proof}
We put $$w_i=h_{ii1},\,\,\,\,\,
w=\sum\limits_{i\neq1}w_i^2+\frac{1}{3}w_1^2,$$
$$f=\sum\limits_{i\neq1}(\lambda_1^2-4\lambda_1\lambda_i)w_i^2-\lambda_1^2w_1^2.$$
If $\lambda_1^2-4\lambda_1\lambda_i\leq[2+\eta(n)]S$, for any
$i\neq1$, we have $f\leq[2+\eta(n)]Sw$. Otherwise, without loss of
generality, we assume that $\lambda_1^2-4\lambda_1\lambda_2=tS$ for
some $t\in[2, \frac{1+\sqrt{17}}{2}]$, and $w_1=zw_2$. Since $H$ is
a constant, we have $(w_1+w_2)^2=(\sum\limits_{i\neq1,2}w_i)^2$.
Hence
\begin{equation}\label{3.40}
\sum\limits_{i\neq1,2}w_i^2\geq\frac{(1+z)^2}{n-2}w_2^2,
\end{equation}
\begin{equation}\label{3.41}
\lambda_1^2\geq(t-2)S.
\end{equation}
From (\ref{3.40}), (\ref{3.41}) and Lemma \ref{lem3}, we have
\begin{eqnarray}\label{3.42}
f-[2+\eta(n)]Sw&\leq&\nonumber[t-2-\eta(n)]Sw_2^2+[t-r-2-\eta(n)]S\sum\limits_{i\neq1,2}w_i^2\\
& &\nonumber-(t-2+\frac{2+\eta(n)}{3})Sw_1^2\\
&\leq&\nonumber[t-2-\eta(n)]Sw_2^2+\frac{t-r-2-\eta(n)}{n-2}(1+z)^2Sw_2^2\\
& &-(t-2+\frac{2+\eta(n)}{3})z^2Sw_2^2,
\end{eqnarray}
where $r=\frac{16t-8-12\sqrt{-2t^2+2t+8}}{17}$. Set
$K(n,t,z)=t-2-\eta(n)+\frac{t-r-2-\eta(n)}{n-2}(1+z)^2-(t-2+\frac{2+\eta(n)}{3})z^2$.
(\ref{3.42}) becomes
\begin{equation}\label{3.43}
f-[2+\eta(n)]Sw\leq K(n,t,z)Sw_2^2.
\end{equation}
Noting that $\frac{\partial K(n,t,z)}{\partial z}|_{z=z_0}=0$, we
have
\begin{equation}\label{3.44}
K(n,t,z)\leq K(n,t,z_0)=\frac{h_n(t)}{51L(n,t)},
\end{equation}
where $L(n,t)=r+2+\eta(n)-t+(n-2)(t-2+\frac{2+\eta(n)}{3})$, and
$h_n(t)$ is defined as in Proposition \ref{prop1} (ii). This
together with Proposition \ref{prop1} (ii) implies
\begin{equation}\label{3.45}
f\leq[2+\eta(n)]Sw.
\end{equation}
Similarly, for any fixed $j$, we have
\begin{eqnarray}\label{31}
f_j&=&\nonumber\sum\limits_{i\neq
j}(\lambda_j^2-4\lambda_j\lambda_i)h_{iij}^2-\lambda_j^2h_{jjj}^2\\
&\leq&(2+\eta(n))S(\sum\limits_{i\neq
j}h_{iij}^2+\frac{1}{3}h_{jjj}^2).
\end{eqnarray}
Hence
\begin{eqnarray}\label{32}
3(A-2B)&=&\nonumber\sum\limits_{i,j,k
\,\,distinct}[2(\lambda_i^2+\lambda_j^2+\lambda_k^2)-(\lambda_i+\lambda_j+\lambda_k)^2]h_{ijk}^2\\
&
&\nonumber-3\sum\limits_{i}\lambda_i^2h_{iii}^2+3\sum\limits_{j}\sum\limits_{i\neq
j}(\lambda_j^2-4\lambda_i\lambda_j)h_{iij}^2\\
&\leq&\nonumber2S\sum\limits_{i,j,k
\,\,distinct}h_{ijk}^2+3\sum\limits_{j}\sum\limits_{i\neq
j}[(\lambda_j^2-4\lambda_i\lambda_j)h_{iij}^2-\lambda_j^2h_{jjj}^2]\\
&\leq&\nonumber(2+\eta(n))S(\sum\limits_{i,j,k
\,\,distinct}h_{ijk}^2+3\sum\limits_{j}\sum\limits_{i\neq
j}h_{iij}^2+\sum\limits_{j}h_{jjj}^2)\\
&=&(2+\eta(n))S\sum\limits_{i,j,k}h_{ijk}^2.
\end{eqnarray}
This proves Lemma \ref{lem4}.

\end{proof}

For the higher dimensional cases, we obtain the following estimate
of $A-2B$.
\begin{lemma}\label{lem2}
If $n\geq6$ and $n\leq S\leq \frac{16}{15}n$, then
$$3(A-2B)\leq [S+4+C_3(n)G^{1/3}+q_5(n,H)]|\nabla h|^2.$$
Here
$C_3(n)=\Big(\frac{3-\sqrt{6}-4p}{\sqrt{6}-1+13p}(6-\sqrt{6}-13p)^2\Big)^{1/3}$,
$p=\frac{1}{13(n-2)}$, $q_5(n,H)=\\
\Big\{2(\frac{16n}{15})^{2/3}(\frac{\sqrt{15n}}{2n-2})^{1/3}
p_{1}^{1/3}+(\frac{64n}{15}-4)^{2/3}(\frac{2\sqrt{15n}}{n-1}+\frac{n^2H}{n-2}+\frac{2n}{n-2}\sqrt{\frac{32n}{15}})^{1/3}\Big\}H^{1/3}$
and
$p_1=5+\frac{3}{2}(\sqrt{\frac{15n-16}{5n-16}}-1)+\frac{3}{2}(\sqrt{\frac{15n-16}{5n-16}}-1)^{-1}$.
\end{lemma}
\begin{proof}
\par For fixed distinct $i,j,k,$ we put
\begin{eqnarray}\label{3.7}
\varphi&=
&\nonumber\lambda_i^2+\lambda_j^2+\lambda_k^2-2\lambda_i\lambda_j-2\lambda_i\lambda_k-2\lambda_j\lambda_k,\\
\psi&=&\lambda_j^2-4\lambda_i\lambda_j.
\end{eqnarray}
If $\lambda_i, \lambda_j, \lambda_k$ have the same sign, it is easy
to see that $\varphi\leq S$. Without loss of generality, we suppose
\begin{equation}\label{3.8}
\lambda_i\lambda_j\leq\lambda_j\lambda_k\leq0,
\lambda_i\lambda_k\geq0.
\end{equation}
Putting $\lambda_i=-x\lambda_j, \lambda_k=-y\lambda_j, x\geq
y\geq0$, we have
\begin{eqnarray}\label{3.9}
\varphi&=&\nonumber\lambda_{i}^{2}+\lambda_{j}^{2}+\lambda_{k}^{2}+2(x+y-xy)\lambda_{j}^{2}\\
&\leq&S+4+2[x\lambda_{j}^{2}-1+(1-x)y\lambda_{j}^{2}-1].
\end{eqnarray}
Setting $a=x\lambda_{j}^{2}-1, b=y\lambda_{j}^{2}-1,
c=(1-x)y\lambda_{j}^{2}-1$, we rewrite (\ref{3.9}) as follow
\begin{equation}\label{3.10}
\varphi\leq S+4+2(a+c).
\end{equation}
Since
\begin{equation}\label{3.11}
\lambda_{j}^{2}+\frac{1}{n-1}(nH-\lambda_{j})^2\leq
S\leq\frac{16n}{15},
\end{equation}
we have
\begin{eqnarray}\label{3.12}
\lambda_{j}^{2}&\leq&\nonumber\frac{16(n-1)}{15}+2H\lambda_j-nH^2\\
&\leq&\frac{16(n-1)}{15}+2H\sqrt{\frac{16n}{15}}.
\end{eqnarray}
Hence
\begin{equation}\label{3.13}
\frac{15}{16(n-1)}\lambda_{j}^{2}\leq1+q_1,
\end{equation}
where $q_1=q_1(n,H):=\frac{\sqrt{15n}}{2(n-1)}H$.
\par When $c\geq0$, we get $x\leq1$ and $a\geq b\geq c\geq0$. Then we
have
\begin{eqnarray}\label{3.14}
c&\leq&\nonumber x(1-x)\lambda_{j}^{2}+q_1-\frac{15}{16(n-1)}\lambda_{j}^{2}\\
&\leq&\nonumber\frac{4n-19}{32n-17}(x+1)^2\lambda_{j}^{2}+q_1\\
&\leq&(1-\frac{16}{5n})\frac{2n-2}{16n-1}(x+1)^2\lambda_{j}^{2}+q_1.
\end{eqnarray}
Similarly, we have
\begin{eqnarray}\label{3.15}
a&\leq&\nonumber x\lambda_{j}^{2}+q_1-\frac{15}{16(n-1)}\lambda_{j}^{2}\\
&\leq&\frac{4n-4}{16n-1}(x+1)^2\lambda_{j}^{2}+q_1.
\end{eqnarray}
and
\begin{eqnarray}\label{3.16}
b&\leq&\nonumber y\lambda_{j}^{2}+q_1-\frac{15}{16(n-1)}\lambda_{j}^{2}\\
&\leq&\frac{4n-4}{16n-1}(y+1)^2\lambda_{j}^{2}+q_1.
\end{eqnarray}
So,
\begin{eqnarray}\label{3.17}
(a+c)^3&=&\nonumber a^3+c^3+3(a^2c+ac^2)\\
&\leq&\nonumber a^3+c^3+3(a^2c+\frac{\epsilon}{2}a^2c+\frac{1}{2\epsilon}c^3)\\
&\leq&\nonumber a^3+b^3+3\Big\{(1+\frac{\epsilon}{2})a^2\Big[(1-\frac{16}{5n})\frac{2n-2}{16n-1}(x+1)^2\lambda_{j}^{2}+q_1\Big]\\
& &\nonumber+\frac{1}{2\epsilon}b^2\Big[\frac{4n-4}{16n-1}(y+1)^2\lambda_{j}^{2}+q_1\Big]\Big\}\\
&=&\nonumber
a^3+b^3+3\frac{2n-2}{16n-1}\Big[(1+\frac{\epsilon}{2})a^2(1-\frac{16}{5n})(x+1)^2\\
& &
+\frac{1}{\epsilon}b^2(y+1)^2\Big]\lambda_{j}^{2}+\Big[3(1+\frac{\epsilon}{2})a^2+\frac{3}{2\epsilon}b^2\Big]q_1,
\end{eqnarray}
for $\epsilon>0$. Letting
$(1+\frac{\epsilon}{2})(1-\frac{16}{5n})=\frac{1}{\epsilon}$, we
have $\epsilon=\sqrt{\frac{15n-16}{5n-16}}-1$. Since
\begin{eqnarray}\label{3.18}
G&\geq&\nonumber 2(\lambda_i-\lambda_j)^2(\lambda_i\lambda_j+1)^2+2(\lambda_j-\lambda_k)^2(\lambda_j\lambda_k+1)^2\\
&=&2[(x+1)^2a^2+(y+1)^2b^2]\lambda_{j}^2,
\end{eqnarray}
we have
\begin{eqnarray}\label{3.19}
(a+c)^3&\leq&\nonumber a^3+b^3+3\frac{2n-2}{(16n-1)\epsilon}[a^2(x+1)^2+b^2(y+1)^2]\lambda_{j}^{2}\\
& &\nonumber+\Big[3(1+\frac{\epsilon}{2})a^2+\frac{3}{2\epsilon}b^2\Big]q_1\\
&\leq&\nonumber
a^2\Big[\frac{4n-4}{16n-1}(x+1)^2\lambda_{j}^{2}+q_1\Big]
+b^2\Big[\frac{4n-4}{16n-1}(y+1)^2\lambda_{j}^{2}+q_1\Big]\\
&
&\nonumber+3\frac{2n-2}{(16n-1)\epsilon}[a^2(x+1)^2+b^2(y+1)^2]\lambda_{j}^{2}\\
& &\nonumber
+\Big[3(1+\frac{\epsilon}{2})a^2+\frac{3}{2\epsilon}b^2\Big]q_1\\
&=&\nonumber\frac{2n-2}{16n-1}(2+\frac{3}{\epsilon})[a^2(x+1)^2+b^2(y+1)^2]\lambda_{j}^{2}\\
& &\nonumber+\Big[(4+\frac{3\epsilon}{2})a^2+(1+\frac{3}{2\epsilon})b^2\Big]q_1\\
&\leq&\frac{n-1}{16n-1}(2+\frac{3}{\epsilon})G+q_2,
\end{eqnarray}
where
$q_2=q_2(n,H,a,b)=\Big[(4+\frac{3\epsilon}{2})a^2+(1+\frac{3}{2\epsilon})b^2\Big]q_1$.
\par When $c<0$ and $a\geq0$, (\ref{3.19}) becomes
\begin{equation}\label{3.19-2}
(a+c)^3\leq a^3\leq\frac{n-1}{16n-1}(2+\frac{3}{\epsilon})G+q_2.
\end{equation}
This yields
\begin{eqnarray}\label{3.20}
a+c&\leq&\nonumber (\frac{n-1}{16n-1}(2+\frac{3}{\epsilon})G+q_2)^{1/3}\\
&\leq&
\Big[\frac{n-1}{16n-1}(2+\frac{3}{\epsilon})G\Big]^{1/3}+q_{2}^{1/3}.
\end{eqnarray}
So,
\begin{eqnarray}\label{3.21}
\varphi&\leq&\nonumber S+4+2\Big[\frac{n-1}{16n-1}(2+\frac{3}{\epsilon})G\Big]^{1/3}+2q_{2}^{1/3}\\
&=& S+4+C^{'}_{3}(n)G^{1/3}+2q_{2}^{1/3}.
\end{eqnarray}
Here $C^{'}_{3}(n)=\frac{8(n-1)}{16n-1}(2+\frac{3}{\epsilon})$.
\par When $c<0$ and $a<0$, it's obvious that (\ref{3.21}) holds.

\par To derive an upper bound for $\psi$, it's sufficient to estimate $\psi$ on $T=\{x\in M;\,\,\psi>S+4\}$.
At a fixed point $x\in T$, we have $\psi>S+4$. This implies
$\lambda_i\lambda_j<0$. Thus, there exists $t>0$, such that
$\lambda_i=-t\lambda_j$. Then we get
\begin{eqnarray}\label{3.22}
S&\geq&\nonumber\lambda_{i}^{2}+\lambda_{j}^{2}+\frac{1}{n-2}(nH-\lambda_{i}-\lambda_{j})^2\\
&=&\frac{n-1}{n-2}(\lambda_{i}^{2}+\lambda_{j}^{2})+\frac{2}{n-2}\lambda_{i}\lambda_{j}
+\frac{n^2H^2}{n-2}-\frac{2nH}{n-2}(\lambda_i+\lambda_j).
\end{eqnarray}
Hence
\begin{eqnarray}\label{3.23}
\psi&\leq&\nonumber
S-\frac{n-1}{n-2}\lambda_{i}^{2}-\frac{1}{n-2}\lambda_{j}^{2}-\frac{2}{n-2}\lambda_{i}\lambda_{j}
-r_1-4\lambda_{i}\lambda_{j}\\
&=&S+\Big(-\frac{n-1}{n-2}t^2+\frac{4n-6}{n-2}t-\frac{1}{n-2}\Big)\lambda_{j}^{2}-r_1,
\end{eqnarray}
where
$r_1=r_1(n,H,\lambda_i,\lambda_j)=\frac{n^2H^2}{n-2}-\frac{2nH}{n-2}(\lambda_i+\lambda_j)$.
From (\ref{3.13}), we have
\begin{eqnarray}\label{3.24}
\psi&\leq&\nonumber
S+4+\Big(-\frac{n-1}{n-2}t^2+\frac{4n-6}{n-2}t\\
& &\nonumber-\frac{1}{n-2}\Big)\lambda_{j}^{2}-\frac{15}{4n-4}\lambda_{j}^{2}+4q_1-r_1\\
&\leq&\nonumber
S+4+\Big(-\frac{n-1}{n-2}t^2+\frac{4n-6}{n-2}t\\
& &-\frac{4}{n-2}\Big)\lambda_{j}^{2}+r_2.
\end{eqnarray}
Here $r_2=4q_1-r_1$. Since
\begin{equation}\label{3.25}
\psi\leq
S-\lambda_{i}^{2}-4\lambda_{i}\lambda_{j}=S+(4t-t^2)\lambda_{j}^{2}
\end{equation}
and
\begin{equation}\label{3.26}
-\frac{n-1}{n-2}t^2+\frac{4n-6}{n-2}t-\frac{4}{n-2}\leq\Big(t-\frac{12}{13(n-2)}\Big)(4-t),
\end{equation}
we obtain
\begin{eqnarray}\label{3.27}
(\psi-S-4)^3&\leq&\nonumber[(4t-t^2)\lambda_{j}^{2}-4]^2
\Big[\Big(-\frac{n-1}{n-2}t^2\\
& &\nonumber+\frac{4n-6}{n-2}t-\frac{4}{n-2}\Big)\lambda_{j}^{2}+r_2\Big]\\
&\leq&\nonumber[(4t-t^2)\lambda_{j}^{2}-(4-t)]^2\Big(t-\frac{12}{13(n-2)}\Big)(4-t)\lambda_{j}^{2}\\
& &\nonumber+[(4t-t^2)\lambda_{j}^{2}-4]^2r_2\\
&\leq&\Big(t-\frac{12}{13(n-2)}\Big)(4-t)^3(t\lambda_{j}^{2}-1)^2\lambda_{j}^{2}+r_3.
\end{eqnarray}
Here $r_3=[(4t-t^2)\lambda_{j}^{2}-4]^2|r_2|$. From (\ref{3.25}), we
get
\begin{equation}\label{3.28}
(4t-t^2)\lambda_{j}^{2}>4.
\end{equation}
Hence $4t-t^2>0$, i.e., $0<t<4$. We define an auxiliary function
$$\xi(t):=\frac{(t-12p)(4-t)^3}{(1+t)^2},\,\,0<t<4,$$ where
$p=\frac{1}{13(n-2)}$. Letting $\frac{d\xi(t)}{dt}|_{t=t_0}=0$, we
get
$$t_0=\sqrt{9p^2+54p+6}+3p-2,$$ and
$$\xi_{0}:=\xi(t_0)=\frac{(2-2t_0+18p)(4-t_0)^2}{1+t_0}.$$ It is
easy to see that $\frac{d^2\xi(t)}{dt^2}|_{t=t_0}<0$. Hence $t_0$ is
the maximum point of $\xi(t)$. Since $t_0\geq\sqrt{6}+10p-2+3p$ and
$\zeta(t)=\frac{(2-2t+18p)(4-t)^2}{1+t}$ is decreasing in $t$, we
have
\begin{equation}\label{3.30}
\xi(t)\leq\xi_0\leq\frac{(6-2\sqrt{6}-8p)(6-\sqrt{6}-13p)^2}{\sqrt{6}-1+13p}=2C_3(n)^3.
\end{equation}
Here
$C_3(n)=\Big[\frac{(6-2\sqrt{6}-8p)(6-\sqrt{6}-13p)^2}{\sqrt{6}-1+13p}\Big]^{1/3}$.
By the definition of $G$, we have
\begin{equation}\label{3.31}
G\geq2(\lambda_i-\lambda_j)^2(1+\lambda_i\lambda_j)^2=2(t+1)^2\lambda_{j}^{2}(t\lambda_{j}^{2}-1)^2.
\end{equation}
From (\ref{3.27}) and (\ref{3.30}), we obtain
\begin{eqnarray}\label{3.32}
(\psi-S-4)^3&\leq&\nonumber\xi(t)(1+t)^2(t\lambda_{j}^{2}-1)^2\lambda_{j}^{2}+r_3\\
&\leq&\nonumber C_3(n)^3G+r_3\\
&\leq& (C_3(n)G^{1/3}+r_{3}^{1/3})^3,
\end{eqnarray}
i.e., $$\psi-S-4\leq C_3(n)G^{1/3}+r_{3}^{1/3}.$$

Therefore, at any point $x\in M$, (\ref{3.32}) holds.

\par By a direct computation, we have $C_3(7)\geq\lim\limits_{n\rightarrow\infty}C^{'}_{3}(n)$ and
$C^{'}_{3}(6)\leq C_3(6)$. Noting that $C^{'}_{3}(n)$ and $C_3(n)$
are both increasing in $n$, we get $C_3(n)\geq C^{'}_{3}(n)$, for
$n\geq6$.
\par Now we give proper upper bounds for $r_3$ and $q_2$. Since $4t-t^2\leq4t-t^2|_{t=2}=4$, we have
\begin{equation}\label{3.33}
(4t-t^2)\lambda_{j}^{2}-4\leq4\times\frac{16n}{15}-4,
\end{equation}
and
\begin{equation}\label{3.34}
|\lambda_i+\lambda_j|\leq\sqrt{2(\lambda_{i}^{2}+\lambda_{j}^{2})}\leq\sqrt{\frac{32n}{15}}.
\end{equation}
Then we obtain
\begin{eqnarray}\label{3.35}
r_3&=&\nonumber [(4t-t^2)\lambda_{j}^{2}-4]^2|r_2|\\
&\leq&\nonumber [(4t-t^2)\lambda_{j}^{2}-4]^2\Big[4q_1+\frac{n^2H^2}{n-2}+\frac{2nH}{n-2}|\lambda_i+\lambda_j|\Big]\\
&\leq& q_3,
\end{eqnarray}
where
$q_3=q_3(n,H)=(\frac{64n}{15}-4)^2\Big(4q_1+\frac{n^2H^2}{n-2}+\frac{2n}{n-2}\sqrt{\frac{32n}{15}}H\Big)$.
Since $x\geq y\geq0$, we have
\begin{equation}\label{3.36}
a^2=(x\lambda_{j}^{2}-1)^2\leq
x^2\lambda_{j}^4=\lambda_{i}^2\lambda_{j}^2\leq\Big(\frac{16}{15}n\Big)^2,
\end{equation}
and
\begin{equation}\label{3.37}
b^2=(y\lambda_{j}^{2}-1)^2\leq
y^2\lambda_{j}^4=\lambda_{k}^2\lambda_{j}^2\leq\Big(\frac{16}{15}n\Big)^2.
\end{equation}
Hence
\begin{equation}\label{3.38}
q_2=\Big[(4+\frac{3\epsilon}{2})a^2+(1+\frac{3}{2\epsilon})b^2\Big]q_1\leq
q_4,
\end{equation}
where $q_4=q_4(n,H)=\Big(\frac{16}{15}n\Big)^2p_1q_1$,
$p_1=5+\frac{3\epsilon}{2}+\frac{3}{2\epsilon}$. Thus, we have
\begin{eqnarray}\label{3.39}
3(A-2B)&\leq&\nonumber
\sum\limits_{i,j,k\,distinct}h_{ijk}^2(\lambda_{i}^2+\lambda_{j}^2+\lambda_{k}^2
-2\lambda_{i}\lambda_{j}-2\lambda_{j}\lambda_{k}-2\lambda_{i}\lambda_{k})\\
& &\nonumber+\sum\limits_{i\neq j}h_{iij}^2(\lambda_{j}^2-4\lambda_{i}\lambda_{j})\\
&\leq&\nonumber \sum\limits_{i,j,k\,distinct}h_{ijk}^2(S+4+C_3(n)G^{1/3}+2q_{2}^{1/3})\\
& &\nonumber+\sum\limits_{i\neq j}h_{iij}^2(S+4+C_3(n)G^{1/3}+r_{3}^{1/3})\\
&\leq& (S+4+C_3(n)G^{1/3}+q_5)|\nabla h|^2,
\end{eqnarray}
where $q_5=q_5(n,H)=2q_{4}^{1/3}+q_{3}^{1/3}$. This completes the
proof of Lemma \ref{lem2}.
\end{proof}

\section{Proof of Main Theorem}\label{sec4}

\par Let $x_n$ be the unique positive solution
of the following equation
$$\frac{n^3}{2(n-1)}x+\frac{n(n-2)}{2(n-1)}\sqrt{n^2x^2+4(n-1)x}=\frac{n}{15}-\frac{n}{23}.$$
Then we have
$$\frac{n^3}{2(n-1)}H^2+\frac{n(n-2)}{2(n-1)}\sqrt{n^2H^4+4(n-1)H^2}\leq\frac{n}{15}-\frac{n}{23},$$
for $H\leq\gamma_1(n):=\sqrt{x_n}$.
\par Now we are in a position to prove our main theorem.
\begin{proof}[Proof of Main Theorem]
For the low dimensional cases ($2\leq n\leq5$), it follows from
(\ref{2.9}), (\ref{2.11}) and (\ref{3.6}) that
\begin{equation}\label{4.30}
\int_{M}[(S-2n-\frac{3}{2})|\nabla
h|^2+\frac{3}{2}(A-2B)+\frac{9}{8}|\nabla S|^2-3nHC]dM\geq0,
\end{equation}
for $H\leq\gamma_1(n)$. By the definition of $F(\Phi)$, we get
\begin{eqnarray}\label{4.21}
\frac{F(\Phi)}{\Phi}&=&\nonumber\frac{\Phi^{2}-n\Phi-nH^2\Phi-nH\bar{f}_{3}}{\Phi}\\
&=&\nonumber\Phi-n-nH^2-\frac{nH\bar{f}_{3}}{\Phi}\\
&\geq&\nonumber\Phi-\beta_0(n,H)-nH^2-\frac{nH}{\Phi}\frac{n-2}{\sqrt{n(n-1)}}\Phi^{3/2}\\
&\geq&\nonumber\Phi-\beta_0(n,H)-nH^2-\frac{n(n-2)}{\sqrt{n(n-1)}}\sqrt{\frac{16}{15}n}H\\
&=&\Phi-\beta_0-q_6,
\end{eqnarray}
where $q_6=q_6(n,H)=nH^2+\frac{4n(n-2)}{\sqrt{15(n-1)}}H$.
Similarly, we have
\begin{equation}\label{4.21-1}
\frac{F(\Phi)}{\Phi}\leq\Phi-\beta_0+D_0+q_6,
\end{equation}
where $D_0=D_0(n,H)=\beta_0(n,H)-n$. On the other hand, we have
\begin{equation}\label{4.31}
\frac{\bar{f}_3}{\Phi}\leq\frac{n-2}{\sqrt{n(n-1)}}\sqrt{\Phi}\leq\frac{n-2}{\sqrt{n(n-1)}}\sqrt{\frac{16n}{15}}.
\end{equation}
It follows from (\ref{4.21-1}) and (\ref{4.31}) that
\begin{eqnarray}\label{4.32}
\Phi F(\Phi)&=&\nonumber(F(\Phi)+n\Phi+nH^2\Phi+nH\bar{f}_3)\frac{F(\Phi)}{\Phi}\\
&\leq&
\Big(\Phi-\beta_0+D_0+q_6+n+nH^2+\frac{4n(n-2)H}{\sqrt{15(n-1)}}\Big)F(\Phi).
\end{eqnarray}
Since
\begin{equation}\label{4.19}
|C|\leq \sqrt{S}|\nabla h|^2,\,\,n\leq S,
\end{equation}
we have
\begin{eqnarray}\label{4.20}
-3nH\int_MCdM&\leq&\nonumber3nH\int_M\sqrt{S}|\nabla
h|^2dM\\
&\leq&3\sqrt{n}H\int_MS|\nabla h|^2dM.
\end{eqnarray}
Here $C=\sum\limits_{i,j,k}h_{ijk}^2\lambda_{i}$. Combining
(\ref{3.2}), (\ref{4.30}), (\ref{4.32}), (\ref{4.20}) and the
condition $\beta(n,H)\leq S\leq\beta(n,H)+\frac{n}{23}$, we obtain
\begin{eqnarray}\label{4.33}
0&\leq&\nonumber\int_M \Big\{\frac{9}{4}(\Phi-\beta_0)F(\Phi)+\Big[\frac{9}{4}\Big(D_0+q_6+n+nH^2+\frac{4n(n-2)H}{\sqrt{15(n-1)}}\Big)\\
& &\nonumber +S-2n-\frac{3}{2}-\frac{9}{4}\Phi+3\sqrt{n}HS\Big]|\nabla h|^2+\frac{3}{2}(A-2B)\Big\}dM\\
&\leq&\nonumber\int_M \Big\{\Big[\frac{9}{4}\Big(\frac{n}{23}+n+D_0+q_6+nH^2+\frac{4n(n-2)H}{\sqrt{15(n-1)}}\Big)\\
& &\nonumber +S-2n-\frac{3}{2}-\frac{9}{4}\Phi+\frac{16n\sqrt{n}H}{5}\Big]|\nabla h|^2+\frac{3}{2}(A-2B)\Big\}dM\\
&=&\nonumber\int_M
\Big[\Big(\frac{n}{4}+\frac{9n}{92}-\frac{5S}{4}-\frac{3}{2}+q_7\Big)\\
& &\times|\nabla h|^2+\frac{3}{2}(A-2B)\Big]dM,
\end{eqnarray}
where
$q_7=q_7(n,H)=\frac{9}{4}\Big(D_0+q_6+2nH^2+\frac{4n(n-2)H}{\sqrt{15(n-1)}}\Big)+\frac{16n\sqrt{n}H}{5}$.
\par When $2\leq n\leq3$, substituting (\ref{2.12}) into
(\ref{4.33}), we get
\begin{eqnarray}\label{4.34}
0&\leq&\nonumber\int_M
\Big[\frac{n}{4}+\frac{9n}{92}+\Big(\frac{\sigma}{2}-\frac{5}{4}\Big)\Big(n+D_0+nH^2+\frac{n}{23}\Big)\\
& &\nonumber-\frac{3}{2}+q_7\Big]|\nabla h|^2dM\\
&\leq&\nonumber\int_M
\Big[\frac{3}{46}(24\sigma-67)+q_7\\
& &+\Big(\frac{\sigma}{2}-\frac{5}{4}\Big)(D_0+nH^2)\Big]|\nabla
h|^2dM.
\end{eqnarray}
There exists a positive constant $\gamma_2(n)$ depending only on $n$
such that
$q_7+\Big(\frac{\sigma}{2}-\frac{5}{4}\Big)(D_0+nH^2)\leq0.1$ for
$H\leq\gamma_2(n)$, which implies that the coefficient of the
integral in (\ref{4.34}) is negative. This together with
(\ref{4.34}) implies $\nabla h=0$.
\par When $4\leq n\leq5$, by Lemma \ref{lem4}, we have
$$3(A-2B)\leq[2+\eta(n)]S|\nabla h|^2,$$ where $\eta(4)=0.16$ and
$\eta(5)=0.23$. Similarly, we get
\begin{eqnarray}\label{4.35}
0&\leq&\nonumber\int_M
\Big[\frac{n}{4}+\frac{9n}{92}+\Big(\frac{2+\eta}{2}-\frac{5}{4}\Big)n
-\frac{3}{2}+q_7\Big]|\nabla h|^2dM\\
&=&\nonumber\int_M \Big(\frac{9n}{92}+\frac{n\eta}{2}
-\frac{3}{2}+q_7\Big)|\nabla h|^2dM\\
&\leq&\int_M \Big(\frac{45}{92}+\frac{5\times0.23}{2}
-\frac{3}{2}+q_7\Big)|\nabla h|^2dM.
\end{eqnarray}
We choose a positive constant $\gamma_2(n)$ depending only on $n$
such that $q_7\leq0.1$ for $H\leq\gamma_2(n)$. Therefore, the
coefficient of the integral in (\ref{4.35}) is negative. This
together with (\ref{4.35}) implies $\nabla h=0$.

\par For higher dimensional cases ($n\geq6$), we define
$u_{ijkl}=\frac{1}{4}(h_{ijkl}+h_{lijk}+h_{klij}+h_{jkli})$. Notice
that $u_{ijkl}$ is symmetric in $i,j,k,l$. Then we have
\begin{eqnarray}\label{4.13}
\sum\limits_{i,j,k,l}(h_{ijkl}^{2}-u_{ijkl}^{2})&=&
\nonumber\frac{1}{16}\sum\limits_{i,j,k,l}[(h_{ijkl}-h_{lijk})^2+(h_{ijkl}-h_{klij})^2\\
& &\nonumber +(h_{ijkl}-h_{jkli})^2+(h_{lijk}-h_{klij})^2\\
& &\nonumber +(h_{lijk}-h_{jkli})^2+(h_{klij}-h_{jkli})^2]\\
&\geq&\nonumber\frac{6}{16}\sum\limits_{i\neq j}[(h_{ijij}-h_{jiji})^2+(h_{jiji}-h_{ijij})^2]\\
&=&\frac{3}{4}G,
\end{eqnarray}
\begin{equation}\label{4.14}
\sum\limits_{i,j,k,l}u_{ijkl}^{2}\geq\sum\limits_{i}u_{iiii}^{2}+3\sum\limits_{i\neq
j}u_{iijj}^2.
\end{equation}
Since $\sum\limits_{i}\mu_{i}=0$, $\sum\limits_{i}\mu_{i}^{2}=\Phi$
and $\sum\limits_{i,j}(\mu_{i}+\mu_{j})u_{iijj}=-F(\Phi)$, using
Proposition \ref{prop2}, we obtain
\begin{equation}\label{4.15}
|\nabla^2 h|^2\geq \frac{3}{4}G+\frac{3F^2(\Phi)}{2(n+4)\Phi}.
\end{equation}
For $0<\theta<1$, $H\leq\gamma_1(n)$ and $\beta(n,H)\leq
S\leq\beta(n,H)+\delta(n)$, we have
\begin{equation}\label{4.16}
\int_M|\nabla^2
h|^2dM\geq\Big(\frac{3(1-\theta)}{4}+\frac{3\theta}{4}\Big)\int_M
GdM+\int_M \frac{3F^2(\Phi)}{2(n+4)\Phi}dM.
\end{equation}
Combining (\ref{2.9}), (\ref{3.6}) and (\ref{4.16}), using Lemma
\ref{lem2} and Young's inequality, we drive the following
inequality.
\begin{eqnarray}\label{4.17}
&
&\nonumber\frac{3(1-\theta)}{4}\int_MGdM+\int_M \frac{3F^2(\Phi)}{2(n+4)\Phi}dM\\
&\leq&\nonumber\int_M\Big[(S-2n-3)|\nabla
h|^2+\frac{3}{2}|\nabla S|^2\\
& &\nonumber+3(A-2B)-3nHC-\frac{3\theta}{4}G\Big]dM\\
&=&\nonumber\int_M(S-2n-3+\frac{3\theta}{2})|\nabla
h|^2dM+(3-\frac{3\theta}{2})\int_M(A-2B)dM\\
& &\nonumber+(\frac{3}{2}-\frac{3\theta}{8})\int_M|\nabla
S|^2dM-3nH\int_MCdM\\
&\leq&\nonumber\int_M(S-2n-3+\frac{3\theta}{2})|\nabla
h|^2dM\\
&
&\nonumber+(1-\frac{\theta}{2})\int_M\Big(S+4+C_3(n)G^{1/3}+q_5\Big)|\nabla
h|^2dM\\
& &\nonumber+(\frac{3}{2}-\frac{3\theta}{8})\int_M|\nabla
S|^2dM-3nH\int_MCdM\\
&\leq&\nonumber\int_M\Big[(2-\frac{\theta}{2})S-2n+1-\frac{\theta}{2}+(1-\frac{\theta}{2})q_5\Big]|\nabla
h|^2dM\\
& &\nonumber+\frac{3(1-\theta)}{4}\int_MGdM+C_4\int_M|\nabla
h|^3dM\\
& &+(\frac{3}{2}-\frac{3\theta}{8})\int_M|\nabla
S|^2dM-3nH\int_MCdM,
\end{eqnarray}
where
$C_4=C_4(n,\theta)=\frac{4}{9}C_3(n)^{3/2}(1-\frac{\theta}{2})^{3/2}(1-\theta)^{-1/2}$.
\par From (\ref{3.1}), we get
\begin{equation}\label{4.18-1}
\frac{1}{2}|\nabla\phi|\Delta\Phi=-F(\Phi)|\nabla\phi|+|\nabla\phi|^3.
\end{equation}
This together with the divergence theorem and Cauchy-Schwarz's
inequality implies
\begin{eqnarray}\label{4.18}
\int_M|\nabla h|^3dM
&\leq&\nonumber\int_MF(\Phi)|\nabla\phi|dM+\epsilon\int_M|\nabla^2\phi|^2dM\\
& &+\frac{1}{16\epsilon}\int_M|\nabla\Phi|^2dM,
\end{eqnarray}
for $\epsilon>0$. From (\ref{4.21}), we have
\begin{equation}\label{4.22}
\int_M \frac{3F^2(\Phi)}{2(n+4)\Phi}dM\geq\int_M
\frac{3}{2(n+4)}(\Phi-\beta_0-q_6)F(\Phi)dM.
\end{equation}
Combining (\ref{4.20}), (\ref{4.17}), (\ref{4.18}) and (\ref{4.22}),
we obtain
\begin{eqnarray}\label{4.23}
0&\leq&\nonumber\int_M\Big[(2-\frac{\theta}{2}+3\sqrt{n}H)S-2n+1-\frac{\theta}{2}+(1-\frac{\theta}{2})q_5\Big]|\nabla
h|^2dM\\
&
&\nonumber+C_4\Big(\int_MF(\Phi)|\nabla\phi|dM+\epsilon\int_M|\nabla^2\phi|^2dM+\frac{1}{16\epsilon}\int_M|\nabla\Phi|^2dM\Big)\\
& &+(\frac{3}{2}-\frac{3\theta}{8})\int_M|\nabla\Phi|^2dM-\int_M
\frac{3}{2(n+4)}(\Phi-\beta_0-q_6)F(\Phi)dM.
\end{eqnarray}
It follows from (\ref{2.9}) and (\ref{2.12}) that
\begin{eqnarray}\label{4.24}
& &\nonumber\int_M|\nabla^2 h|^2dM\\
&\leq&\nonumber\int_M[(\sigma+1+3\sqrt{n}H)S-2n-3]|\nabla
h|^2dM\\
& &+\frac{3}{2}\int_M|\nabla S|^2dM.
\end{eqnarray}
From (\ref{3.2}), we get
\begin{equation}\label{4.25}
\int_M\frac{1}{2}|\nabla\Phi|^2dM=\int_M(\Phi-\beta_0)F(\Phi)dM+\int_M(\beta_0-\Phi)|\nabla\phi|^2dM.
\end{equation}
Thus, (\ref{4.23}) becomes
\begin{eqnarray}\label{4.26}
0&\leq&\nonumber\int_M\Big\{\Big[2-\frac{\theta}{2}+3\sqrt{n}H+\epsilon C_4(\sigma+1+3\sqrt{n}H)\Big](\Phi-\beta_0+\beta)\\
& &\nonumber-2n+1-\frac{\theta}{2}+(1-\frac{\theta}{2})q_5-\epsilon
C_4(2n+3)\Big\}|\nabla\phi|^2dM\\
&
&\nonumber+\Big(\frac{3}{2}-\frac{3\theta}{8}+\frac{C_4}{16\epsilon}+
\frac{3\epsilon C_4}{2}\Big)\int_M|\nabla
S|^2dM\\
& &\nonumber+C_4\int_MF(\Phi)|\nabla\phi|dM-\int_M \frac{3}{2(n+4)}(\Phi-\beta_0-q_6)F(\Phi)dM\\
&=&\nonumber\int_M\Big\{D\Big[2-\frac{\theta}{2}+3\sqrt{n}H+\epsilon C_4(\sigma+1+3\sqrt{n}H)\Big]+1-\frac{\theta}{2}(n+1)\\
&
&\nonumber+3n\sqrt{n}H+(1-\frac{\theta}{2})q_5+\frac{3}{2(n+4)}q_6+\epsilon
C_4(\sigma
n+3n^{\frac{3}{2}}H-n-3)\Big\}\\
&
&\nonumber\times|\nabla\phi|^2dM+C_4\int_MF(\Phi)|\nabla\phi|dM-\Big[1-\frac{\theta}{4}+\frac{C_4}{8\epsilon}-3\sqrt{n}H\\
& &\nonumber+\epsilon
C_4(2-\sigma-3\sqrt{n}H)\Big]\int_M(\Phi-\beta_0)|\nabla\phi|^2dM\\
& &+\Big(3-\frac{3\theta}{4}+\frac{C_4}{8\epsilon}+3\epsilon
C_4-\frac{3}{2(n+4)}\Big)\int_M(\Phi-\beta_0)F(\Phi)dM,
\end{eqnarray}
where $D=D(n,H)=\beta(n,H)-n$. On the other hand, it follows from
(\ref{4.21-1}) that
\begin{eqnarray}\label{4.27}
C_4\int_MF(\Phi)|\nabla\phi|dM
&\leq&\nonumber2C_4(\beta_0+\delta)\epsilon\int_MF(\Phi)dM\\
& &\nonumber+\frac{C_4}{8(\beta_0+\delta)\epsilon}\int_MF(\Phi)|\nabla\phi|^2dM\\
&\leq&\nonumber2C_4(\beta_0+\delta)\epsilon\int_MF(\Phi)dM\\
& &\nonumber+\int_M\frac{C_4\Phi}{8(\beta_0+\delta)\epsilon}(\Phi-\beta_0+D_0+q_6)|\nabla\phi|^2dM\\
&\nonumber\leq&2C_4(n+D_0+\delta)\epsilon\int_MF(\Phi)dM\\
&
&+\int_M\frac{C_4}{8\epsilon}(\Phi-\beta_0+D_0+q_6)|\nabla\phi|^2dM.
\end{eqnarray}
Substituting (\ref{4.27}) into (\ref{4.26}), we get
\begin{eqnarray}\label{4.28}
0&\leq&\nonumber\int_M\Big\{D\Big[2-\frac{\theta}{2}+3\sqrt{n}H+\epsilon C_4(\sigma+1+3\sqrt{n}H)\Big]\\
&
&\nonumber+1-\frac{\theta}{2}(n+1)+3n\sqrt{n}H+(1-\frac{\theta}{2})q_5+\frac{3}{2(n+4)}q_6\\
& &\nonumber+\epsilon C_4(\sigma
n+3n^{\frac{3}{2}}H-n-3)\Big\}|\nabla\phi|^2dM+2C_4(n+D_0+\delta)\epsilon\\
&
&\nonumber\times\int_MF(\Phi)dM+\int_M\frac{C_4}{8\epsilon}(\Phi-\beta_0+D_0+q_6)|\nabla\phi|^2dM-\Big[1-\frac{\theta}{4}\\
& &\nonumber+\frac{C_4}{8\epsilon}-3\sqrt{n}H+\epsilon
C_4(2-\sigma-3\sqrt{n}H)\Big]\int_M(\Phi-\beta_0)|\nabla\phi|^2dM\\
&
&\nonumber+\Big(3-\frac{3\theta}{4}+\frac{C_4}{8\epsilon}+3\epsilon
C_4-\frac{3}{2(n+4)}\Big)\int_M(\Phi-\beta_0)F(\Phi)dM\\
&\leq&\nonumber\int_M\Big\{D\Big[2-\frac{\theta}{2}+3\sqrt{n}H+\epsilon C_4(\sigma+1+3\sqrt{n}H)\Big]+3n^{\frac{3}{2}}\epsilon C_{4}H\\
&
&\nonumber+3n\sqrt{n}H+(1-\frac{\theta}{2})q_5+\frac{3}{2(n+4)}q_6+2C_4D_0\epsilon+\frac{C_4}{8\epsilon}(D_0+q_6)\\
& &\nonumber+1-\frac{\theta}{2}(n+1)+\epsilon C_4(\sigma
n+n-3+5\delta)+\frac{C_4\delta}{8\epsilon}\\
&
&\nonumber+\Big(3-\frac{3\theta}{4}-\frac{3}{2(n+4)}\Big)\delta+3\sqrt{n}\delta
H+3\epsilon C_{4}\sqrt{n}\delta H\Big\}|\nabla\phi|^2dM
\\
& &-\Big[1-\frac{\theta}{4}+\epsilon
C_4(2-\sigma)\Big]\int_M(\Phi-\beta_0)|\nabla\phi|^2dM.
\end{eqnarray}
Let $\epsilon=\sqrt{\frac{\delta}{8(\sigma n+n-3+5\delta)}}$ and
$\theta=0.86$. Then
$$C_4(n)=\frac{4}{9}\times0.57^{3/2}\times0.14^{-1/2}\times\sqrt{\frac{3-\sqrt{6}-4p}{\sqrt{6}-1+13p}}(6-\sqrt{6}-13p),$$
where $p=\frac{1}{13(n-2)}\leq\frac{1}{13}$. Notice that $C_4(n)$ is
increasing in $n$. Thus, we have
$C_4(n)\leq\lim\limits_{l\rightarrow\infty}C_4(l)\approx1.1185<1.1186$,
$0<\epsilon\leq\frac{1}{6}$ and $\sigma=\frac{\sqrt{17}+1}{2}$. So,
$0.785+C_4\epsilon(2-\sigma)>0$. Let $\delta(n)=\frac{n}{23}$. There
exists a positive constant $\gamma_2(n)$ depending only on $n$ such
that $q_8\leq0.001$, for $H\leq\gamma_2(n)$. Here
$q_8=q_8(n,H)=D\Big[2-\frac{\theta}{2}+3\sqrt{n}H+\epsilon
C_4(\sigma+1+3\sqrt{n}H)\Big]+3n^{\frac{3}{2}}\epsilon C_{4}H
+3n\sqrt{n}H+(1-\frac{\theta}{2})q_5+\frac{3}{2(n+4)}q_6+2C_4D_0\epsilon+\frac{C_4}{8\epsilon}(D_0+q_6)
+3\sqrt{n}\delta H+3\epsilon C_{4}\sqrt{n}\delta H$. Then we have
\begin{eqnarray}\label{4.29}
0&\leq&\nonumber\Big[1-\frac{\theta}{2}(n+1)+\epsilon C_4(\sigma
n+n-3+5\delta)+\frac{C_4\delta}{8\epsilon}\\
&
&\nonumber+\Big(3-\frac{3\theta}{4}-\frac{3}{2(n+4)}\Big)\delta+q_8\Big]\int_M|\nabla\phi|^2dM\\
&\leq&\nonumber\Big(-0.43n+0.57+C_4\sqrt{\frac{\delta}{2}(\sigma
n+n-3+5\delta)}\\&
&+(2.355-\frac{3}{2(n+4)})\delta+0.001\Big)\int_M|\nabla \phi|^2dM.
\end{eqnarray}
\par When $n\geq75$, we have
\begin{eqnarray}\label{4.29-1}
& &\nonumber C_4\sqrt{\frac{\delta}{2}(\sigma
n+n-3+5\delta)}\\
&\leq&\nonumber 1.1186\times\sqrt{\frac{n}{46}(\sigma
n+n+\frac{5n}{23})}\\
&\leq&0.3207n.
\end{eqnarray}
Hence
\begin{eqnarray}\label{4.29-2}
0&\leq&\nonumber\Big[-0.43n+0.571+0.3207n\\
& &+\Big(2.355-\frac{3}{2(n+4)}\Big)\frac{n}{23}\Big]\int_M|\nabla
\phi|^2dM.
\end{eqnarray}
Since $\frac{3}{2(n+4)}\frac{n}{23}\geq0.061$, we have
\begin{equation}\label{4.29-3}
0\leq\Big(-0.43n+0.571+0.3207n+\frac{2.355n}{23}-0.061\Big)\int_M|\nabla
\phi|^2dM.
\end{equation}
Notice that the coefficient of the integral in (\ref{4.29-3}) is
negative. This shows that $\nabla \phi=0$.
\par When $6\leq n\leq74$, we note that
$C_4(n)=\frac{4}{9}\times0.57^{3/2}\times0.14^{-1/2}\times\sqrt{\frac{3-\sqrt{6}-4p}{\sqrt{6}-1+13p}}(6-\sqrt{6}-13p)$,
$p=\frac{1}{13(n-2)}$ and $\delta=\frac{n}{23}$. This implies that
the coefficient of the integral in (\ref{4.29}) is negative. Thus,
we have $\nabla \phi=0$.
\par
So, we conclude that $$\nabla \phi=\nabla h=0,\,\, \mbox{\ for\ }
\,\, n\geq 2 \,\, \mbox{\ and\ } \,\, H\leq \gamma(n),$$ where
$\gamma(n)=\min\{\gamma_1(n),\gamma_2(n)\}$. This together with
(\ref{3.3}) and (\ref{3.4}) implies that $F(\Phi)=0$ and
$\Phi=\beta_0(n,H)$.

\par When $H=0$, $\Phi=\beta_0(n,H)$ becomes $S=n$, i.e.,  $M$ is one of the Clifford torus\,
$$\mathbb{S}^{k}(\sqrt{\frac{k}{n}})\times
\mathbb{S}^{n-k}(\sqrt{\frac{n-k}{n}}), \,1\le k\le n-1.$$

\par When $H\neq0$, using Proposition \ref{prop1} (i), we get
$$\lambda_1=\dots=\lambda_{n-1}=H-\sqrt{\frac{\beta(n,H)-nH^2}{n(n-1)}},$$
$$\lambda_n=H+\sqrt{\frac{(n-1)(\beta(n,H)-nH^2)}{n}}.$$
Therefore, $M$ is the Clifford torus
$$\mathbb{S}^{1}(\frac{1}{\sqrt{1+\mu^2}})\times \mathbb{S}^{n-1}(\frac{\mu}{\sqrt{1+\mu^2}})$$
in ${S}^{n+1}$, where $\mu=\frac{nH+\sqrt{n^2H^2+4(n-1)}}{2}$.

This completes the proof of Main Theorem.
\end{proof}
\par Finally we would like to propose the following conjectures.

\begin{cconja}
Let $M$ be a closed hypersurface with constant mean curvature $H$ in $\mathbb{S}^{4}$. We have\\
(i) If $\beta(3,H)\leq S\leq 6+9H^2$, then $S\equiv\beta(3,H)$ or
$S\equiv 6+9H^2$, i.e., $M$ is a Clifford torus or a tube
of the Veronese surface.\\
(ii) If $S\geq 6+9H^2$, then $S\equiv 6+9H^2$, i.e., $M$ is a tube
of the Veronese surface.
\end{cconja}
\par In particular, if $H=0$, the problem above is still open. In the case where $M$ is a closed hypersurface with constant mean
curvature and constant scalar curvature in $\mathbb{S}^{4}$,
Conjecture A was solved by Almeida-Brito \cite{AB} and Chang
\cite{C2}. It is well known that the possible values of the squared
length of the second fundamental forms of all closed isoparametric
hypersurfaces with constant mean curvature $H$ in the unit sphere
form a discrete set $I(\subset\mathbb{R}).$ The following conjecture
can be viewed as a general version of the Chern conjecture.
\begin{cconjb}  Let $M$ be an
$n$-dimensional closed hypersurface with constant mean curvature $H$
in the unit sphere $\mathbb{S}^{n+1}$. \\
(i) Assume that $a<b$ and $[a,b]\bigcap I=\{a,b\}$. If $a\leq S \leq
b$, then $S\equiv a$ or $S\equiv b,$  and $M$ is an isoparametric
hypersurface in $\mathbb{S}^{n+1}$.\\
(ii) Set $c=\sup_{t\in I} t.$ If $S\geq c,$ then $S\equiv c$, and
$M$ is an isoparametric hypersurface in $\mathbb{S}^{n+1}$.
 \end{cconjb}
\par In particular, if $a=nH^2$ and $b=\alpha(n,H)$, Theorem
C provides an affirmative answer to Conjecture B (i).

\bibliographystyle{amsplain}

\end{document}